\documentclass[12pt]{amsart}

\usepackage{amsmath}
\usepackage{amssymb}
\usepackage{amsthm}
\usepackage[all]{xy}
\usepackage{url}
\usepackage[margin=1.00in]{geometry}

\newcommand{\B}{\mathbf B}
\newcommand{\Bm}{\mathbf B_-}
\newcommand{\bI}{\mathbf I}

\newcommand{\cO}{\mathcal O}
\renewcommand{\P}{\mathbb P}
\newcommand{\Q}{\mathbb Q}
\newcommand{\R}{\mathbb R}
\newcommand{\cX}{\mathcal X}
\newcommand{\rat}{\dashrightarrow}

\newcommand{\abs}[1]{\left\lvert #1 \right\rvert}
\newcommand{\set}[1]{\left\{ #1 \right\}}
\newcommand{\floor}[1]{\left\lfloor #1 \right\rfloor}

\newtheorem{theorem}{Theorem}[section]
\newtheorem{corollary}[theorem]{Corollary}
\newtheorem{lemma}[theorem]{Lemma}

\theoremstyle{definition}
\newtheorem{definition}[theorem]{Definition}
\newtheorem{remark}[theorem]{Remark}
\newtheorem*{remark*}{Remark}
\newtheorem*{proposition*}{Proposition}

\DeclareMathOperator{\Aut}{Aut}
\DeclareMathOperator{\Eff}{Eff}
\DeclareMathOperator{\Effb}{\overline{Eff}}
\DeclareMathOperator{\Movb}{\overline{Mov}}

\DeclareMathOperator{\Nef}{Nef}

\DeclareMathOperator{\Bs}{Bs}

\DeclareMathOperator{\Cr}{Cr}
\DeclareMathOperator{\Crb}{{\overline{\Cr}}}

\DeclareMathOperator{\ord}{ord}
\DeclareMathOperator{\PGL}{PGL}

\DeclareMathOperator{\Supp}{Supp}

\newcommand{\bp}{\mathbf p}
\newcommand{\bq}{\mathbf q}
\newcommand{\dinf}{D_\lambda}

\title{The Diminished Base Locus Is Not Always Closed}
\author{John Lesieutre} 
\email{johnl@math.mit.edu}
\address{Department of Mathematics\\
  MIT\\
  77 Massachusetts Avenue\\
  Cambridge, MA 02139, USA} 
\thanks{This research was supported by an NSF Graduate Research
  Fellowship under Grant \#1122374.}

\begin{document}

\begin{abstract}
  We exhibit a pseudoeffective \(\R\)-divisor \(\dinf\) on the blow-up
  of \(\P^3\) at nine very general points which lies in the closed
  movable cone and has negative intersections with a set of curves
  whose union is Zariski dense.  It follows that the diminished base
  locus \(\Bm(\dinf) = \bigcup_{\text{$A$ ample}} \B(\dinf+A)\) is not
  closed and that \(\dinf\) does not admit a Zariski decomposition in
  even a very weak sense.  By a similar method, we construct an
  \(\R\)-divisor on the family of blow-ups of \(\P^2\) at ten distinct
  points, which is nef on a very general fiber but fails to be nef
  over countably many prime divisors in the base.
\end{abstract}

\maketitle

\section{Introduction}

For a pseudoeffective \(\R\)-divisor \(D\) on a normal projective
variety \(Y\), the diminished base locus (also called the non-nef
locus or restricted base locus) is the union
\[
\Bm(D) = \bigcup_{\substack{\text{$A$ ample} \\ \text{$D+A$
      $\Q$-Cartier}}} \B(D+A),
\]
where \(\B(D+A) = \bigcap_{n \geq 1} \Bs(n(D+A))\) is the stable base
locus~\cite{elmnp}.  This is at most a countable union of
subvarieties, but in many examples the union is finite, i.e.\ Zariski
closed.  We will give an example of an \(\R\)-divisor for which this
locus is not Zariski closed.

\begin{theorem}
\label{main}
Let \(X\) be the blow-up of \(\P^3\) at nine very general points.  There
exists a pseudoeffective \(\R\)-divisor \(\dinf\) on \(X\) with the
following properties:
\begin{enumerate}
\item There is a countable set of curves \(C_n \subset X\) with \(\dinf \cdot
  C_n < 0\), whose union is Zariski dense on \(X\).
\item \(\Bm(\dinf)\) is a countable union of curves.
\item There is no decomposition \(f^\ast \dinf
  \equiv_{\mathrm{num}} P + N\) of \(f^\ast \dinf\) into nef and
  effective components on any birational model \(f : Y \to X\).
\end{enumerate}
Further, there exists a big \(\R\)-divisor \(\dinf^\prime\) on
\(\P_X(\cO_X \oplus \cO_X(1))\) for which \(\Bm(\dinf^\prime)\) is a
countable union of curves, where \(\cO_X(1)\) is any very ample line
bundle on \(X\).
\end{theorem}

A similar method gives an example related to the behavior of nefness
of divisors in families.  
\begin{theorem}
\label{neffamily}
Let \(\Sigma = ((\P^2)^{10} \setminus \Delta)/\PGL(3)\), where
\(\Delta\) is the locus where two points coincide, and let \(\cX \to
\Sigma\) be the family whose fiber over \(\bp \in \Sigma\) is
isomorphic to the blow-up of \(\P^2\) at the corresponding ten points.
There exists an \(\R\)-divisor \(C_\lambda\) on \(\cX\) such that
\(C_{\lambda,\bp}\) is nef for very general \(\bp\), but there are
countably many prime divisors \(V_n \subset \Sigma\) such that
\(C_{\lambda,\bp}\) is not nef if \(\bp \in V_n\).
\end{theorem}

The behavior of this example is an instance of the following property
of nefness.
\begin{proposition*}[\cite{lazarsfeld}, Proposition 1.4.14]
  Suppose that \(X\) and \(S\) are varieties over a field and \(\pi :
  X \to S\) is a surjective and proper morphism.  Let \(D\) be an
  \(\R\)-Cartier divisor on \(X\).  If \(D_0\) is nef for some \(0 \in
  S\), then \(D_s\) is nef for very general \(s \in S\) (i.e.\ for all
  \(s\) not contained in some countable union of subvarieties).
\end{proposition*}

There do not seem to be any examples known in characteristic \(0\) in
which \(D\) is a Cartier divisor and nefness is not simply an open
condition.  The example demonstrates that, at least in the generality
of \(\R\)-divisors, the ``very general'' of the conclusion is indeed
essential.  Some recent examples in positive and mixed characteristic
are discussed in~\cite{langer}.

Both examples arise from classical constructions. Throughout, we work
over an uncountable algebraically closed field of arbitrary
characteristic.  Starting with a set of \(k\) very general points on
\(\P^n\), and a divisor of degree \(d\) with given multiplicities at
these points, we make a sequence of Cremona transformations centered
at certain subsets of the points, and compute the degree and
multiplicities of the strict transform of the divisor under these
transformations.  The changes in degrees and multiplicities are
governed by an action of a Coxeter group of type \(T_{2,n+1,k-n-1}\),
an observation originally due to Coble~\cite{coble}. A modern account
of Coble's work can be found in the survey of Dolgachev and
Ortland~\cite{dolgachev}.  If \(n=2\) and \(k = 10\), or \(n=3\) and
\(k = 9\), then the associated groups of type \(T_{2,3,7}\) and
\(T_{2,4,5}\) are infinite and have elements acting with eigenvalues
of norm greater than \(1\).  The divisors of the examples arise as the
corresponding eigenvectors. For simplicity, we make explicit
computations for specific transformations in these groups, but the
results are valid for many other elements as well.

These Coxeter groups have often played a role in the study of the
birational geometry of blow-ups of projective space.  Nagata's
construction of infinitely many \((-1)\)-curves on blow-ups of
\(\P^2\) at \(9\) very general points makes use of the fact that the
group of type \(T_{2,3,6}\) is infinite~\cite{nagata}, while Mukai's
characterization of the blow-ups of \(\P^n\) which are Mori dream
spaces again relies on the finiteness of associated Coxeter
groups~\cite{mukai}. Laface and Ugaglia's study of a
higher-dimensional analogue of the Harbourne-Hirschowitz conjecture
also involves sequences of Cremona transformations centered at various
points~\cite{lu1}.

Elements of these groups have been studied from a dynamical
perspective as well: for blow-ups of \(\P^n\) at special
configurations of points, there can exist (pseudo-)automorphisms whose
action on cohomology is given by these elements.  This was studied in
dimension 2 in the work of McMullen~\cite{mcmullen} and
Bedford-Kim~\cite{bedfordkim2},\cite{bedfordkim}, and in higher
dimension by Perroni-Zhang~\cite{perronizhang}.  Eigenvectors of the
action on cohomology of the automorphisms in~\cite{mcmullen} provide
examples of the sort in Theorem~\ref{neffamily}.  I have learned that
recent work of T.\ Bayraktar also considers the diminished base loci
of a general class of \(\R\)-divisors constructed as eigenvectors of
pseudo-automorphisms; the example presented here is roughly one in
which the inclusion in Theorem 1.1 of~\cite{bayraktar} is an equality,
and the union is infinite.

The next section contains some preliminary lemmas needed for the
constructions. Section~\ref{exampleone} provides the example of
Theorem~\ref{neffamily}.  Section~\ref{cremona} introduces the
standard Cremona transformation \(\Cr : \P^3 \rat \P^3\), leading to
the construction of \(\dinf\).  The various claims of
Theorem~\ref{main} are proved in Sections~\ref{baselocus}
and~\ref{zariski} as Lemmas~\ref{negcurves},
\ref{nowzd}, and \ref{bignonclosed}.

\section{Preliminaries}
\label{prelim}

We first record a simple observation which implies that
\(\R\)-divisors arising as eigenvectors of automorphisms of \(N^1(X)\)
often generate extremal rays on the various cones of divisors.
\begin{lemma}[cf.~\cite{birkhoff}]
\label{coneextremal}
Suppose that \(V\) is a finite dimensional real vector space, \(G
\subset V\) is a closed convex cone with nonempty interior and
containing no line, and \(T : V \to V\) is a linear map with \(T(G)=
G\).  If \(T\) has a real eigenvalue \(\lambda\) of algebraic
multiplicity one, with magnitude larger than that of any other
eigenvalue, then the \(\lambda\)-eigenvector \(v_\lambda\) (with
appropriate sign) spans an extremal ray on \(G\).
\end{lemma}
\begin{proof}
  Fix a norm \(\abs{\cdot}\) on \(V\) and write \(V = \R v_\lambda
  \oplus W\), where \(W\) is the direct sum of the other real Jordan
  blocks, so that \(T\vert_W\) has all eigenvalues with norm strictly
  less than \(\lambda\).  Since \(G\) has nonempty interior, there
  exists \(v \in G\) with nonzero component in the
  \(v_\lambda\)-eigenspace.  Then \(\frac{1}{\lambda^n} T^n v\)
  converges to some nonzero multiple of \(v_\lambda\).  Switching the
  sign if needed, we conclude that \(v_\lambda\) is contained in
  \(G\).

  Suppose that \(v_\lambda\) is not extremal, i.e.\ that there exists
  a nonzero \(w \in W\) for which \(v_\lambda + w\) and \(v_\lambda
  -w\) are both in \(G\).  Since its image contains an open set, \(T\)
  is invertible and \(T^{-1}(G) = G\). There is a sequence \(n_i\) for
  which \(T^{-n_i} w/\abs{T^{-n_i} w}\) converges to a nonzero limit
  \(r \in V\).  Since \(T\vert_W\) has eigenvalues less than
  \(\lambda\), \(\abs{\lambda^n T^{-n} w}\) grows without bound as
  \(n\) increases, and \(v_\lambda/\abs{\lambda^n T^{-n} w}\)
  converges to \(0\).  It follows that the two sequences of vectors in
  \(G\)
\[
\frac{\lambda^{n_i} T^{-n_i}(v_\lambda \pm w)}{\abs{\lambda^{n_i}
    T^{-n_i} w}} = \frac{v_\lambda}{\abs{\lambda^{n_i} T^{-n_i} w}}
\pm \frac{\lambda^{n_i} T^{-n_i} w}{\abs{\lambda^{n_i} T^{-n_i} w}}
\]
converge to \(\pm r\).  The closedness of \(G\) implies that both
\(r\) and \(-r\) are contained in \(G\), contradicting the assumption
that \(G\) contains no line.
\end{proof}

Both examples deal with blow-ups of projective space, and it will be
useful to establish some basic notation for \(k\)-tuples of points on
\(\P^n\).  Throughout, we work over an uncountable algebraically
closed field of arbitrary characteristic.  Let \(\Sigma = ((\P^n)^k
\setminus \Delta)/\PGL(n+1)\) be the set of \(k\)-tuples with all
points distinct, and let \(\pi : \cX \to \Sigma\) be the family whose
fiber \(X_{\bp}\) over \(\bp = (p_1,\ldots,p_k) \in \Sigma\) is
isomorphic to the blow-up of \(\P^n\) at the corresponding \(k\)
points.

If \(Y\) is a normal projective variety, the group of \(\R\)-Cartier
divisors on \(Y\) modulo numerical equivalence is denoted \(N^1(Y)\),
and \([D] \in N^1(Y)\) is the numerical class of a divisor \(D\),
though when no confusion is possible we omit the brackets. Dually,
\(N_1(Y)\) is the group of curves modulo numerical equivalence, and
the class of \(C\) is written \([C]\).

For any fiber \(X = X_\bp\), there are decompositions \(N^1(X) = \R H
\oplus \bigoplus_{i=1}^k \R E_i\) and \(N_1(X) = \R h \oplus
\bigoplus_{i=1}^k \R e_i\), where \(H\) is the pullback of the
hyperplane class on \(\P^n\), \(h\) is the class of the strict
transform of a line disjoint from the points of \(\bp\), \(E_i\) are
the exceptional divisors, and \(e_i\) the classes of lines in the
\(E_i\).  We will refer to \(H,E_1,\ldots,E_k\) and \(h,e_1,\ldots,e_k\) as
the standard bases for \(N^1(X)\) and \(N_1(X)\).
The intersection pairing on the exceptional classes is given
by \(E_i \cdot e_i = -1\).  If \(\bp\) and \(\bq\) are two different
sets of points, there is an isomorphism \(\Phi_{ \bp \bq} :
N^1(X_{\bp}) \to N^1(X_{\bq})\) which sends \(H_{\bp}\) to \(H_{\bq}\)
and \(E_{i,\bp}\) to \(E_{i,\bq}\); the matrix for \(\Phi_{\bp \bq}\)
with respect to the above bases is the \((k+1) \times (k+1)\) identity
matrix.

The pseudoeffective cone \(\Effb(X) \subset N^1(X)\) is the closure of
the cone \(\Eff(X)\) generated by classes of effective Cartier
divisors, and the movable cone \(\Movb(X) \subset N^1(X)\) is the
closure of the cone generated by classes of such divisors whose base
locus has codimension at least \(2\). The next lemma shows that if
\(\bp\) and \(\bq\) are very general, the movable and pseudoeffective
cones of \(X_{\bp}\) and \(X_{\bq}\) coincide under the identification
\(\Phi_{\bp\bq}\).

\begin{lemma}
\label{points}
There is a set \(U \subset \Sigma\), the complement of a countable
union of subvarieties, such that if \(\bp\) and \(\bq\) lie in \(U\),
then \(\Phi_{\bp \bq}(\Eff(X_{\bp})) = \Eff(X_{\bq})\), \(\Phi_{\bp
  \bq}(\Effb(X_{\bp})) = \Effb(X_{\bq})\), and \(\Phi_{\bp
  \bq}(\Movb(X_{\bp})) = \Movb(X_{\bq})\).
\end{lemma}
\begin{proof}
  Since every divisor class on \(X_\bp\) or \(X_\bq\) is the
  restriction of a class on \(\cX\), to check that the effective (resp.\
  movable) cones coincide, it is enough to check that for very general
  \(\bp\) and \(\bq\), precisely the same integral classes \(D\) on
  \(\cX\) have effective (resp.\ movable) restrictions to \(X_\bp\) and
  \(X_\bq\).

  For a given class \(D = dH - \sum m_i E_{i}\) on \(\cX\), the set of
  \(\bp\) for which \(h^0(X_\bp,D_\bp) > 0\) is a closed subset of
  \(\Sigma\) by the semicontinuity theorem. It follows that \(X_\bp\)
  and \(X_\bq\) have the same effective integral classes as long as
  these two points lie off of the countably many proper closed subsets
  that arise in this way, and so the effective and pseudoeffective
  cones coincide.

  For the movable cone, we again restrict our attention to integral
  classes.  An integral class \(D_\bp\) on \(X_\bp\) has base locus of
  codimension \(2\) if \(h^0(X_\bp,D_\bp) > 1\) and \(D_\bp\) has no
  fixed part, i.e.\ there does not exist a nonzero class \(F_\bp\)
  with \(h^0(X_\bp,F_\bp) > 0\) such that \(h^0(X_\bp,D_\bp-F_\bp) =
  h^0(X_\bp,D_\bp)\).  The result follows as above by the fact that
  for any integral \(D\) and \(F\) on \(X\), each of
  \(h^0(X_\bp,F_\bp)\), \(h^0(X_\bp,D_\bp-F_\bp)\), and
  \(h^0(X_\bp,D_\bp)\) is constant for \(\bp\) in some open set, and
  only countably many \(D\) and \(F\) are considered.
\end{proof}

We will use the following properties of the diminished base locus,
which follow from the definition (see~\cite{elmnp} for details).
\begin{lemma}
\label{bminusbasic}
Suppose that \(D\) is a pseudoeffective \(\R\)-divisor on a normal projective
variety \(Y\).
\begin{enumerate}
\item \(\Bm(D)\) depends only on the numerical class \([D] \in
  N^1(Y)\).
\item \(\Bm(D) = \emptyset\) if and only if \(D\) is nef.
\item If \(C\) is a curve with \(D \cdot C < 0\), then \(C \subset
  \Bm(D)\).
\item \(\Bm(D+D^\prime) \subseteq \Bm(D) \cup \Bm(D^\prime)\).
\item If \(f : Y^\prime \to Y\) is a surjective morphism between
  smooth varieties, \(\Bm(f^\ast D) = f^{-1}(\Bm(D))\).
\item If \(\set{A_i}\) is a sequence of ample divisors converging to
  \(0\) in \(N^1(Y)\), with each \(D + A_i\) a \(\Q\)-divisor, then
  \(\Bm(D) = \bigcup_j \B(D+A_j)\).
\item If \(D \in \Movb(X)\), then every component of \(\Bm(D)\) has
  codimension at least \(2\).
\end{enumerate}
\end{lemma}

\section{Nefness in families of $\R$-divisors}
\label{exampleone}

In this section, we adopt the notation of Section~\ref{prelim} for
blow-ups of \(\P^2\) at \(k = 10\) points.  The proof of
Theorem~\ref{neffamily} is contained in Lemmas~\ref{msigmatwodim},
~\ref{vidistinct}, and ~\ref{notnef}.

The standard Cremona transformation \(\Cr : \P^2 \rat \P^2\) given by
\([X_1,X_2,X_3] \mapsto [X_1^{-1},X_2^{-1},X_3^{-1}]\) has a
resolution
\[
\xymatrix{
X \ar[d]_\pi \ar@{=}[r] & X^\prime \ar[d]^{\pi^\prime} \\
\P^2 \ar@{-->}[r]^{\Cr} & \P^2
}
\]
Here \(\pi\) is the blow-up of \(\P^2\) at three points, and
\(\pi^\prime\) contracts the strict transforms of the lines between
any two of those points.  We employ the two notations \(X\) and
\(X^\prime\) to emphasize that the standard bases
\(\set{h,e_1,e_2,e_3}\) and
\(\set{h^\prime,e_1^\prime,e_2^\prime,e_3^\prime}\) are different.  If
\(C\) is any curve on \(X\), then its class on \(X^\prime\) in the new
basis is given by \(M([C])\) where
\[
M = \left( \begin{array}{rrrr}
2 & 1 & 1 & 1 \\
-1 & 0 & -1 & -1 \\
-1 & -1 & 0 & -1 \\
-1 & -1 & -1 & 0 \\
\end{array} \right).
\]

If \(\bp \in \Sigma\) is a configuration with  \(p_8\),
\(p_9\), and \(p_{10}\) in linear general position, there is a Cremona
transformation \(\Cr_\bp : \P^2 \rat \P^2\) defined by \(g^{-1} \circ
\Cr \circ g\), where \(g\) an element of \(\Aut(\P^2)\) sending
\(p_8\), \(p_9\), \(p_{10}\) to the points \([1,0,0]\), \([0,1,0]\),
\([0,0,1]\), and \(\Cr\) is the standard Cremona transformation.  Let
\(\rho : \Sigma \rat \Sigma\) be the rational map given by \(
(p_1,\ldots,p_{10}) \mapsto
(p_8,p_9,p_{10},\Cr_\bp(p_1),\ldots,\Cr_\bp(p_7))\).  This map is
regular off of the set \(L \subset \Sigma\), defined as the locus of
\(\bp\) with some \(p_i\) on a line through two of \(p_8\), \(p_9\),
and \(p_{10}\).  Let \(\bp\) be any point of \(\Sigma \setminus L\),
and set \(\bq = \rho(\bp)\).  Write \(\Pi_\sigma\) for the permutation
matrix for \(\sigma = (8,9,10,1,2,3,4,5,6,7)\), and consider the map
\(M_\sigma^{\bp \bq} : N^1(X_\bp) \to N^1(X_\bq)\) given in the
standard bases by
\[
M_\sigma^{\bp \bq} =  \left( \begin{array}{c|c}
M & 0 \\\hline
0 & I_7
\end{array}
\right)
\left( \begin{array}{c|c}
1 & 0 \\\hline
0 & \Pi_\sigma
\end{array}
\right),
\]
where both of these are \(11 \times 11\) block matrices, but with
different block sizes.  If \(C\) is a curve on \(X_\bp\), there is a
curve on \(X_{\bq}\) lying in the class \(M^{\bp \bq}_\sigma([C])\),
obtained by cyclically reordering the points so the last three come
first, and then taking the strict transform of \(C\) under a Cremona
transformation centered at these points.  If \(\bp\) is in very
general position, then the map \( \Phi_{\bp \bq}^{-1} : N^1(X_\bq) \to
N^1(X_\bp)\) is an isomorphism which identifies the effective cones,
by Lemma~\ref{points}.  Let \(M_\sigma = \Phi_{\bp \bq}^{-1} \circ
M_\sigma^{\bp \bq} : N^1(X_\bp) \to N^1(X_\bp)\) be the composition.

For very general \(\bp\), since both \(\Phi_{\bp \bq}^{-1}\) and \(
M_\sigma^{\bp \bq}\) identify the effective cones, we have
\(M_\sigma(\Effb(X_\bp)) = \Effb(X_\bp)\).  Because \(M_\sigma\)
preserves the intersection form on \(N^1(X_\bp)\), it satisfies
\(M_\sigma(\Nef(X_\bp)) = \Nef(X_\bp)\) as well.  

If \(\bp\) is a point for which the effective cone of \(X_{\bp}\) is
larger than that of a very general configuration, it is not
necessarily the case that \(\Phi_{\bp \bq}^{-1}\) maps effective
classes to effective classes. The divisor of the example will
fail to be nef precisely over certain configurations \(\bp\) with the
first three points collinear, the first six on a conic, etc.  These
are ``nodal relations'' among the points, in the terminology of
McMullen~\cite{mcmullen}.

\begin{lemma}
\label{msigmatwodim}
The map \(M_\sigma\) has characteristic polynomial
\((t-1)t^5q(t+t^{-1})\), where \(q(t) = t^5-t^4-6t^3+5t^2+8t-5\).
\(M_\sigma\) has a unique eigenvalue \(\lambda \approx 1.431\) of
magnitude greater than \(1\). When the \(\lambda\)-eigenvector
\(C_{\lambda,\bp}\) is written as \(h - \sum_{i=1}^{10} r_i e_i\), the
first three coefficients satisfy \(r_1+r_2+r_3 > 1\). 

The divisor \(C_{\lambda,\bp}\) is nef on \(X_{\bp}\) for very general
\(\bp\).
\end{lemma}
\begin{proof}
  The inequality on the coefficients can be checked by computing an
  approximation of the eigenvalue and then expressing each of the
  coefficients as a rational function of \(\lambda\).  The claimed
  nefness then follows from Lemma~\ref{coneextremal}, with the cone
  \(G = \Nef(X_\bp) \subset N^1(X)\) and with \(M_\sigma\) for the
  linear map \(T\).
\end{proof}

\begin{remark}
  In the standard coordinates, the divisor is approximately
\[
C_\lambda \approx (1 , -0.451, -0.440, -0.408, -0.315, -0.307, 
 -0.285, -0.220, -0.215, -0.199, -0.154).
\]
\end{remark}
Let \(C_\lambda\) be the corresponding divisor \(h - \sum_{i=1}^{10}
r_i e_i\) on the total space \(\cX\).  Though \(C_{\lambda,\bp}\) is
nef for very general \(\bp\), we will see that if \(\bp\) lies on any
of countably many subvarieties \(V_n\) of \(\Sigma\) for which
\(X_\bp\) contains \((-2)\)-curves of certain classes,
\(C_{\lambda,\bp}\) is not nef.  Define \(V_0\) to be the set of \(\bp
\in \Sigma\) for which \(p_1\), \(p_2\), and \(p_3\) are collinear.
If \(\bp_0 \in V_0\), there is a curve \(\bar{\ell} \subset X_{\bp_0}\) of
class \( C_0 = h - e_1 - e_2 - e_3\).  Then \(C_{\lambda,{\bp_0}} \cdot
\bar{\ell} = 1 - r_1-r_2-r_3 < 0\), and \(C_{\lambda_,\bp_0}\) is not
nef.  Similarly, \(\bp_1 = \rho(\bp_0)\) is a configuration of points
with the first six lying on a conic, and the strict transform of that
conic on \(X_{\bp_1}\) has negative intersection with
\(C_{\lambda,\bp_1}\).  Generally, for \(n \geq 0\) define \(V_{n+1}
\subset \Sigma\) to be the strict transform of \(V_n\) under \(\rho\).
\begin{lemma}
\label{vidistinct}
Each \(V_n\) is a prime divisor not equal to \(L\), and \(V_m\) and
\(V_n\) are distinct if \(m \neq n\).  For any point \(\bp_n \in
V_n\), there exists a curve \(C_n \subset X_{\bp_n}\) in the
class \(M_\sigma^n(C_0)\).
\end{lemma}
\begin{proof}
  To prove these sets are distinct, we will construct a sequence of
  points \(\bp_n \in V_n \setminus L\) such that \(X_{\bp_n}\)
  contains a curve lying in the class \(M_\sigma^n(C_0)\), which is
  the unique rational curve of self-intersection less than or equal
  \(-2\).  Let \(E \subset \P^2\) be a smooth elliptic curve.
  Construct \(\bp_0 \in V_0\) by choosing points on \(E\) such that
  \(p_1\), \(p_2\), and \(p_3\) are the points of intersection of
  \(E\) with some line \(\ell\) meeting \(E\) transversely, and
  \(p_4,\ldots,p_{10}\) have the property that if \(3d-\sum_{i=1}^{10}
  m_i = 0\), the class \(d \ell\vert_E - \sum_{i=1}^{10} m_i p_i\) is
  not linearly equivalent to \(0\) on \(E\) unless \(m_4 = \cdots =
  m_{10} = 0\). This condition will be met if these points are chosen
  to be very general.  Write \(\bar{\ell}\) and \(\bar{E}\) for the
  strict transforms of \(\ell\) and \(E\) on \(X_{\bp_0}\).

  Suppose that \(C \sim d \, \pi^\ast h - \sum_{i=1}^{10} m_i e_i\) is
  a rational curve with \(K_{X_{\bp_0}} \cdot C \geq 0\).  Since
  \(K_{X_{\bp_0}} \sim -\bar{E}\), we have \(\bar{E} \cdot C \leq 0\),
  and so \(\bar{E} \cdot C = 0\).  Then \(3d - \sum_{i=1}^{10} m_i =
  0\), and the hypothesis on the points implies that \(C \sim h - e_1
  - e_2 - e_3\) is the curve \(\bar{\ell}\).  It follows that 
  under any sequence of Cremona transformations, no three points will
  become collinear; indeed, the strict transform of a line containing
  these three points would be a \((-2)\)-curve on \(X_{\bp_0}\), but
  \(\bar{\ell}\) is the only such.  We may therefore define a sequence
  of points \(\bp_{n+1} \in V_{n+1}\) by taking \(\bp_{n+1} =
  \rho(\bp_n)\).  For each \(n\), the image \(C_n\) of \(\bar{\ell}\)
  is the unique \((-2)\)-curve on \(X_{\bp_n}\), and lies in the class
  \(M_\sigma^n(C_0)\), where \(C_0 = h-e_1-e_2-e_3\).  This implies
  that the divisors \(V_m\) are distinct.

  A general point \(\bp_n \in V_n\) is of the form \(\rho^n(\bp_0)\)
  for some point \(\bp_0 \in V_0\).  The strict transform on
  \(X_{\bp_0}\) of a line through the first three points of \(\bp_0\)
  has class \(C_0\), and this curve has class \(M_\sigma^n(C_0)\) on
  \(X_{\bp_n}\).
\end{proof}

\begin{lemma}
\label{notnef}
If \(\bp \in V_n\), then \(C_{\lambda,\bp}\) is not nef.
\end{lemma}
\begin{proof}
  For any point \(\bp \in V_n\), there is a curve \(C_n \subset
  X_{\bp}\) with class \(M_\sigma^n(C_0)\).  Then
\[
C_{\lambda,\bp} \cdot C_n = \left( \frac{1}{\lambda^n}
  M_\sigma^n ([C_{\lambda,\bp}]) \right) \cdot M_\sigma^n
([C_0]) = \frac{1}{\lambda^n} [C_{\lambda,\bp}] \cdot [C_0]
< 0. \qedhere
\]
\end{proof}

\begin{remark}
\label{weyldisclaimer}
  The matrix \( \left( \begin{smallmatrix} M & 0 \\ 0 & I_7
    \end{smallmatrix}
  \right)\) corresponding to a Cremona transformation, together with
  the matrices \( \left( \begin{smallmatrix} 1 & 0 \\ 0 & P
    \end{smallmatrix} \right)\) which permute the exceptional
  components by a permutation \(P\), define an action of the Coxeter
  group of type \(T_{2,3,7}\) on \(N^1(X_\bp)\).  This action is
  explored more fully in the original work of Coble \cite{coble} and
  the book of Dolgachev and Ortland~\cite{dolgachev}.  Elements of
  this group correspond to a finite sequences of permutations of the
  points followed by Cremona transformations.  Any such product
  preserves the nef and effective cones for very general \(\bp\), and
  so the construction carried out above works just as well for other
  elements of this group with a leading eigenvalue larger than \(1\).
\end{remark}

\section{The standard Cremona transformation and its iterates}
\label{cremona}

We now turn to the second example, and will employ the notation of
Section~\ref{prelim} for blow-ups of \(\P^3\).  Some notation from
Section~\ref{exampleone} will be reused in the new context.  The
example of Theorem~\ref{main} will be constructed as an eigenvector of
a map \(N^1(X) \to N^1(X)\) induced on a blow-up of \(\P^3\) by a
certain sequence of Cremona transformations.

The standard Cremona transformation of \(\P^3\) centered at four
non-coplanar points \(p_1,\ldots,p_4\) is the birational map \(\Cr :
\P^3 \rat \P^3\) defined by \[ [X_1,X_2,X_3,X_4] \mapsto
[X_1^{-1},X_2^{-1},X_3^{-1},X_4^{-1}], \] where the coordinates are
chosen so the points \(p_i\) lie at the intersections of the
coordinate hyperplanes. The map \(\Cr\) is toric and is easily seen to
have a resolution
\[
\xymatrix{
 & Y \ar[dl]_p \ar[dr]^{p^\prime} & \\
X \ar@{-->}[rr]^{\Crb} \ar[d]_\pi &&  X^\prime \ar[d]^{\pi^\prime} \\
\P^3 \ar@{-->}[rr]^{\Cr} && \P^3
}
\]
Here both \(\pi : X \to \P^3\) and \(\pi^\prime : X^\prime \to \P^3\)
are the blow-up of \(\P^3\) at \(p_1,\ldots,p_4\), with exceptional
divisors \(E_i\) and \(E_i^\prime\) respectively.  Let \(F_i\) and
\(F_i^\prime\) denote the strict transforms on \(X\) and \(X^\prime\)
of planes through the three points other than \(p_i\), and \(H\) and
\(H^\prime\) the pullbacks of \(\cO_{\P^3}(1)\).  Take \(l_{ij}\) and
\(l_{ij}^\prime\) to be the lines in \(\P^3\) through \(p_i\) and
\(p_j\), and \(\bar{l}_{ij}\) and \(\bar{l}_{ij}^\prime\) their strict
transforms.  The \(\bar{l}_{ij}\) are smooth rational curves with
normal bundle \(\cO_{\P^1}(-1) \oplus \cO_{\P^1}(-1)\) and \(\Crb\) is
the flop of these six curves.  More precisely, \(p\) is the blow-up of
\(X\) along the six curves \(\bar{l}_{ij}\), with exceptional divisors
isomorphic to \(\P^1 \times \P^1\), and these are contracted along the
other ruling by \(p^\prime\). The strict transform of \(F_i\) under
\(\Crb\) is the exceptional divisor \(E_i^\prime\), while the strict
transform of \(E_i\) is \(F_i^\prime\).

The indeterminacy locus of \(\Crb : X \rat X^\prime\) is the union of
the six curves \(\bar{l}_{ij}\); since this map is an isomorphism in
codimension \(1\), taking strict transforms of divisors induces an
isomorphism \(M : N^1(X) \to N^1(X^\prime)\), as well as an
isomorphism \(\check{M} : N_1(X) \to N_1(X^\prime)\) defined by requiring \(D
\cdot C = MD \cdot \check{M}C\).  This action has been studied by Laface and
Ugaglia in connection with special linear systems of divisors on
\(\P^3\)~\cite{lu1},\cite{lu2}. In that context \(M\) describes the
change in the multiplicity of a divisor at prescribed points under
Cremona transformations centered at those points.

\begin{lemma}
\label{formulai}
The isomorphisms \(M : N^1(X) \to N^1(X^\prime)\) and \(\check{M} : N_1(X) \to
N_1(X^\prime)\) are given in the standard bases
by the matrices
\[
M = \left( \begin{array}{rrrrr}
3 & 1 & 1 & 1 & 1 \\
-2 & 0 & -1 & -1 & -1 \\
-2 & -1 & 0 & -1 & -1 \\
-2 & -1 & -1 & 0 & -1 \\
-2 & -1 & -1 & -1 & 0 \\
\end{array} \right), \quad
\check{M} = \left( \begin{array}{rrrrr}
3 & 2 & 2 & 2 & 2 \\
-1 & 0 & -1 & -1 & -1 \\
-1 & -1 & 0 & -1 & -1 \\
-1 & -1 & -1 & 0 & -1 \\
-1 & -1 & -1 & -1 & 0 \\
\end{array} \right).
\]
\end{lemma}
\begin{proof}
  For every \(i\), we may write \(H \sim F_i + \sum_{j \neq i} E_j\),
  and so \(4 H \sim \sum_{i=1}^4 F_i + 3 \sum_{j=1}^4
  E_j\). Similarly, \(4 H^\prime \sim \sum_{i=1}^4 F_i^\prime + 3
  \sum_{j=1}^4 E_j^\prime\).  Taking strict transforms yields \(4 M(H)
  = \sum_{i=1}^4 E_i^\prime + 3 \sum_{j=1}^4 F_j^\prime \), and \(4
  M(H) - 12 H^\prime = -8 \sum_{j=1}^4 E_i^\prime\).  This gives
  \(M(H) = 3H^\prime - 2 \sum_{j=1}^4 E_j^\prime\), which is the first
  column of \(M\).  For the other columns, write \(M(E_i) = F_i^\prime
  = H^\prime - \sum_{j \neq i} E_j^\prime\).  The matrix for
  \(\check{M}\) is then determined by \(M^t I_{1,4} \check{M} = I_{1,4}\),
  which \(I_{1,4}\) is a \(5 \times 5\) diagonal matrix with diagonal
  entries \((1,-1,-1,-1,-1)\).
\end{proof}

If \(\bp \in \Sigma\) is a set of \(k \geq 4\) points in linear
general position (i.e.\ with no more than three coplanar), there is a
Cremona transformation \(\Cr_\bp : \P^3 \rat \P^3\) defined as
\(g^{-1} \circ \Cr \circ g\), where \(g\) is an automorphism \(\P^3\)
which sends \(p_1,\ldots,p_4\) to the standard coordinate points.
This is well-defined, since the parameter space \(\Sigma\)
parametrizes points only up to automorphism.  This Cremona
transformation centered at the first four points induces a birational
map which is an isomorphism in codimension \(1\), again denoted by
\(\Crb_\bp : X_{\bp} \rat X_{\bq}\), where \(\bq =
(p_1,\ldots,p_4,\Cr_\bp(p_5), \ldots,\Cr_\bp(p_k))\).
\begin{corollary}
\label{formula}
Suppose that \(\bp\) is a \(k\)-tuple in \(\P^3\) with no four
points coplanar and consider the map \(\Crb : X_{\bp} \rat
X_{\bq}\) induced by a standard Cremona transformation centered
at the first four points.  
\begin{enumerate}
\item If \(D\) is any divisor on \(X_{\bp}\), then \[[\Crb_\bp(D)] =
  \left( \begin{array}{c|c} M & 0 \\\hline 0 & I_{k-4} \end{array}
  \right)([D]),\] where \(\Crb_\bp(D)\) denotes the strict transform of
  \(D\).
\item If \(C\) is any curve on \(X_{\bp}\) which does not meet the
  curves \(\bar{l}_{ij}\) which make up the indeterminacy locus of
  \(\Crb\), then \[[\Crb_\bp(C)] = \left( \begin{array}{c|c} \check{M}
      & 0 \\\hline 0 & I_{k-4} \end{array} \right)([C]),\] where
  \(\Crb_\bp(C)\) denotes the strict transform of \(C\).
\end{enumerate}
\end{corollary}
\begin{proof}
  The strict transform of \(E_i\) is \(E_i^\prime\) for \(i > 4\), so
  the coefficients on these divisors are unaffected, and (1) is just
  Lemma~\ref{formulai}.  (2) follows from the fact that if \(C\) is
  disjoint from the indeterminacy locus of \(\Crb\), the intersection
  of \(C\) with a divisor is unchanged under strict transform, and
  \(\left( \begin{smallmatrix} \check{M} & 0 \\ 0 & I_{k-4} \end{smallmatrix}
  \right)\) is the linear map which preserves the intersection form.
\end{proof}

We now focus on the case that \(k=9\) points are blown up.
If \(I\) is a \(4\)-tuple from among the nine points, there is a
birational map \(\Cr_I : \P^3 \rat \P^3\) defined as a standard
Cremona transformation centered at the first four points of \(I\),
inducing a birational map \(\Crb_I : X_\bp \to X_\bq\) which is an
isomorphism in codimension \(1\).  Given a sequence \(\bI =
(I_1,\ldots,I_n)\) of \(4\)-tuples from among the nine points, the
composition \(\Cr_\bI = \Cr_{I_n} \circ \cdots \circ \Cr_{I_1}\) is
not defined in general; four of the points might become coplanar under
some \(\Cr_{I_{j-1}}\).  However, if \(\bp = \bp_0\) is in very
general position, arbitrary compositions of Cremona transformations
are defined.  When the composition is defined, we write \(\Crb_{I_j} :
X_{\bp_{j-1}} \rat X_{\bp_j}\) for the induced birational maps of
the blow-ups, and \(\Crb_{\bI} : X_{\bp_0} \rat X_{\bp_n}\) for their
composition.

If \(\bar{\ell} \subset X_\bp\) is the strict transform of a line
through \(p_1\) and \(p_2\), the numerical class of its strict
transform under \(\Crb_\bI\) could be computed using
Corollary~\ref{formula} if it were known that the strict transform of
\(\bar{\ell}\) under \(\Crb_{I_{k-1}} \circ \cdots \circ \Crb_{I_1}\)
is disjoint from the indeterminacy locus of \(\Crb_{I_k}\) for every
\(k \leq n-1\).  Laface and Ugaglia have shown that this is indeed the
case for very general blow-ups.  The strategy of the proof is to
specialize to the situation where the points lie on a genus \(1\)
curve, and reduce the claimed disjointness to the nonvanishing of
certain combinations of the points in the Picard group.


\begin{theorem}[\cite{lu2}, Proposition 2.7]
\label{iterate}
Let \(\bI = (I_1,\ldots,I_n)\) be a finite sequence of \(4\)-tuples, and
let \(\ell\) be the line in \(\P^3\) between \(p_1\) and \(p_2\), with
\(\bar{\ell}\) its strict transform on \(X=X_{\bp }\).  There
exists an open subset \(U_\bI \subset \Sigma\) such that if \(\bp\)
is contained in \(U_\bI\), the following hold:
\begin{enumerate}
\item The composition \(\Cr_{\bI} : \P^3 \rat \P^3\)
  is well-defined.
\item If \(\bar{\ell}\) is not contained in the indeterminacy locus of
  \(\Crb_{\bI}\), then for each \(1 \leq j \leq n\), the strict
  transform \(\bar{\ell}_{j-1} \subset X_{\bp_{j-1}}\) is disjoint
  from the indeterminacy locus of \(\Crb_{I_j}\).
\end{enumerate}
\end{theorem}

We now consider compositions of Cremona transformations centered at
judiciously chosen sequences of quadruples from the among nine points.
Let \(\sigma \in S_9\) be the permutation \((6,7,8,9,1,2,3,4,5)\), and
take \(I_j = (\sigma^{-j}(1),\ldots,\sigma^{-j}(4))\).
The composition \(\Cr_{I_j} \circ \cdots \circ \Cr_{I_1}\) could
equivalently be realized by repeatedly making a Cremona transformation
centered at \(p_6,\ldots,p_9\) and then cyclically permuting the
indices so these points become \(p_1,\ldots,p_4\).

Let \(X = X_\bp\) be the blow-up at a very general configuration
\(\bp\).  Define \(M_\sigma : N^1(X) \to N^1(X)\) and \(\check{M}_\sigma : N_1(X)
\to N_1(X)\) by
\[
M_\sigma = 
\left( \begin{array}{c|c}
M & 0 \\\hline
0 & I_5
\end{array}
\right)
\left( \begin{array}{c|c}
1 & 0 \\\hline
0 & \Pi_\sigma
\end{array}
\right), \quad \check{M}_\sigma = 
\left( \begin{array}{c|c}
\check{M} & 0 \\\hline
0 & I_5
\end{array}
\right)
\left( \begin{array}{c|c}
1 & 0 \\\hline
0 & \Pi_\sigma
\end{array}
\right),
\]
where \(\Pi_\sigma\) is the permutation matrix for \(\sigma\).  The
class of the strict transform of a divisor \(D\) under \(\Crb_{I_n}
\circ \cdots \circ \Crb_{I_1}\) is \(\left( \begin{smallmatrix} 1 & 0
    \\ 0 &\Pi_\sigma \end{smallmatrix} \right)^{-n} M_\sigma^n([D])\).
Since \(D\) is movable and each \(\Crb_{I_j}\) is an isomorphism in
codimension \(1\), this strict transform is a movable divisor as well.
This strict transform is a divisor on a different blow-up \(X_{\bq}\)
(as in Section~\ref{exampleone}), but if \(\bp\) is very general then
by Lemma~\ref{points} this defines a movable class on \(X\) as well,
and so \(M_\sigma(\Movb(X)) = \Movb(X)\).  Thus \(M_\sigma : N^1(X)
\to N^1(X)\) is a linear map which preserves the effective and movable
cones.  Similarly, if \(C\) is a curve with strict transforms disjoint
from the indeterminacy loci of each \(\Crb_k\), its strict transform
has class \(\left( \begin{smallmatrix} 1 & 0 \\ 0
    &\Pi_\sigma \end{smallmatrix} \right)^{-n} \check{M}_\sigma^n([C])\) by
Corollary~\ref{formula}. The following lemma summarizes the essential
properties of \(M_\sigma\).
\begin{lemma}
\label{eigenstuff}
The linear transformation \(M_\sigma\) has characteristic polynomial
\(p(t) = (t+1)(t-1) \, t^4 q(t+t^{-1})\), where \(q(t) = t^4 - 3t^3 +
4t -1\). \(M_\sigma\) has four real eigenvalues: \(1\), \(-1\),
\(\lambda \approx 1.800\) and \(1/\lambda\).  When the
\(\lambda\)-eigenvector \(D_\lambda\) is written as \(H - \sum r_i
E_i\), the first two coefficients satisfy \(r_1+r_2 > 1\).
\end{lemma}
\begin{proof}
  The claims about the eigenvalues are easily verified from the
  characteristic polynomial.  To obtain the claimed inequality on the
  coefficients, one may express the components of the eigenvector in
  terms of the dominant eigenvalue and compute their approximate
  values.
\end{proof}
\begin{remark}
\label{digits}
To three decimal places, \(\dinf\) is given in components by
\[
\dinf \approx (1, -0.640, -0.634, -0.615, -0.554, -0.355, -0.352, -0.341,
-0.307, -0.197) .
\]
\end{remark}

\begin{remark}
  As in the two-dimensional case of Remark~\ref{weyldisclaimer}, the
  matrix \(\left( \begin{smallmatrix} M & 0 \\ 0 &
      I_{k-4} \end{smallmatrix} \right)\) and permutation matrices
  generate the action of a Coxeter group of type \(T_{2,4,5}\) on
  \(N^1(X_\bp)\).  The eigenvectors of many elements other than the
  \(M_\sigma\) considered above have similar properties, including a
  non-closed diminished base locus.
\end{remark}

\section{The geometry of $\dinf$}
\label{baselocus}

\begin{lemma}
\label{extremal}
  The class \(\dinf\) lies in \(\Movb(X)\) and spans an extremal
  ray on \(\Effb(X)\).
\end{lemma}
\begin{proof}
  The two claims follow from Lemma~\ref{coneextremal} by taking \(V =
  N^1(X)\) and \(T = M_\sigma\), with \(G = \Movb(X)\) and \(G =
  \Effb(X)\) respectively. The hypothesis on the dominant eigenvalue
  is verified in Lemma~\ref{eigenstuff}.
\end{proof}

\begin{lemma}[$=$ Theorem~\ref{main}, (i), (ii)]
\label{negcurves}
If \(\bp\) is very general, there is an infinite set of curves \(C_n
\subset X = X_\bp\) such that \(\dinf \cdot C_n < 0\), and
\(\Bm(\dinf)\) is not closed.  The curves \(C_n\) are Zariski dense on
\(X\).
\end{lemma}
\begin{proof}
  The strategy is to construct curves \(C_n\) in the classes
  \(\check{M}_\sigma^n([C_0])\), where \(C_0\) is a line through \(p_1\) and
  \(p_2\).  By Theorem~\ref{iterate}, we can find a sequence of
  configurations \(\bp_j\), defined for all integers \(j\), with
  \(\bp_0 = \bp\) and such that the maps \(\Crb_{I_j} : X_{\bp_{j-1}}
  \rat X_{\bp_j}\) are defined for all \(j\).  We may additionally
  assume that if \(\ell \subset X_{\bp_j}\) is a line not contained in
  the indeterminacy locus of \(\Crb_{I_{j+1}}\), then for all \(k \geq
  0\) the strict transform of \(\ell\) on \(X_{\bp_{j+k}}\) is
  disjoint from the indeterminacy locus of \(\Crb_{j+k+1}\).

  Suppose that \(\bar{\ell} \subset X_{\bp_{-n}}\) is the strict
  transform of a line between \(p_i\) and \(p_j\).  By
  Theorem~\ref{iterate}, as long as \(p_i\) and \(p_j\) are not among
  the base points of \(\Cr_{I_1}\), the composition \(\Crb_{I_n} \circ
  \cdots \circ \Crb_{I_1} \) is well-defined for all \(n\), and 
  the strict transforms of \(\bar{\ell}\) are disjoint from the
  indeterminacy loci of the maps \(\Crb_{I_j}\). 
  Taking \(\bar{\ell}\) to be the line between \(p_{\sigma^{n}(1)}\)
  and \(p_{\sigma^{n}(2)}\) on \(X_{\bp_{-n}}\), we thus obtain a
  curve \(C_n \subset X\) with class \(\check{M}_\sigma^n([C_0])\), where
  \([C_0] = h - e_1 - e_2\) is the class of a line through the first
  two points. Note that \(\Crb_{I_{-n+1}} : X_{\bp_{-n}} \rat
  X_{\bp_{-n+1}}\) is centered at
  \(p_{\sigma^{n-1}(1)},\ldots,p_{\sigma^{n-1}(4)}\).  Since
  \(\sigma^n(1) = \sigma^{n-1}(5)\) and \(\sigma^n(2) =
  \sigma^{n-1}(6)\), \(\ell\) is not among the curves in the
  indeterminacy locus of \(\Crb_{I_{-n+1}}\).

  The computation of \(\dinf\) in Lemma~\ref{eigenstuff} gives \(\dinf
  \cdot C_0 = 1-(r_1+r_2) < 0\), and so
 \[ \dinf \cdot C_n = (\lambda^{-n} M_\sigma^n \dinf) \cdot
 (\check{M}_\sigma^n C_0) = \lambda^{-n} (\dinf \cdot C_0) < 0. \] 

 By (3) of Lemma~\ref{bminusbasic}, each curve \(C_{n}\) is contained
 in \(\Bm(\dinf)\).  However, \(\dinf\) is movable and so
 \(\Bm(\dinf)\) contains no divisors by (7) of the same lemma.  It
 follows that \(\Bm(\dinf)\) is a countable union of curves.

 We now show that the curves are Zariski dense.  Suppose that \(S
 \subset X\) is any surface, and let \(\psi : \tilde{S} \to X\) be the
 inclusion of a resolution of \(S\).  Since \(D_\lambda\) is movable,
 \(S\) is not contained in the base locus of \(D_\lambda + A\) for any
 ample \(A\), and thus \(\psi^\ast(D_\lambda)\) is pseudoeffective.
 If \(C_n \subset S\), then a curve \(\bar{C}_n \subset \tilde{S}\)
 mapping finitely to \(C_n\) has \((\psi^\ast(D_\lambda) \cdot
 \bar{C}_n)_{\tilde{S}} = (D_\lambda \cdot C_n)_X < 0\).  However, a
 pseudoeffective \(\R\)-divisor on a smooth surface can have negative
 intersection with only finitely many curves, namely those in the
 support of the negative part of its Zariski decomposition (recalled
 in Theorem~\ref{zdecomp}).  Thus only finitely many of the curves
 \(C_n\) are contained in any surface.
\end{proof}

The first few classes \([C_n] = \delta h - \sum_i \mu_i e_i\) are
given below.
\[
\begin{array}{r|r|rrrrrrrrr}
n  & \delta & \mu_1 & \mu_2 & \mu_3 & \mu_4 & \mu_5 & \mu_6 & \mu_7 & 
\mu_8 & \mu_9 \\\hline
0 & 1 & 1 & 1 & 0 & 0 & 0 & 0 & 0 & 0 & 0  \\
1 & 3 & 1 & 1 & 1 & 1 & 1 & 1 & 0 & 0 & 0  \\
2 & 7 & 3 & 2 & 2 & 2 & 1 & 1 & 1 & 1 & 1  \\
3 & 13 & 4 & 4 & 4 & 4 & 3 & 2 & 2 & 2 & 1  \\
4 & 25 & 8 & 8 & 8 & 7 & 4 & 4 & 4 & 4 & 3 \\
5 & 45 & 14 & 14 & 14 & 13 & 8 & 8 & 8 & 7 & 4 
\end{array}
\]

On a given variety \(X\), the set of divisors for which \(\Bm(D)\) is
not closed has measure 0 in \(N^1(X)\); all such classes are
\emph{unstable} in the sense of~\cite{elmnp}.  Nevertheless, one
expects that on ``sufficiently complicated'' varieties there should
exist divisors for which \(\Bm(D)\) is not closed.  The following
gives one result in this direction.
\begin{corollary}
  Suppose that \(Y\) is a normal projective threefold.  There exists a
  finite set of points \(q_1,\ldots,q_j\) on \(Y\) such that if \(r :
  Y^\prime \to Y\) is the blow-up of the \(q_i\), there is an
  \(\R\)-divisor \(D\) on \(Y^\prime\) for which \(\Bm(D)\) is not
  closed.
\end{corollary}
\begin{proof}
  Fix a separable finite map \(s : Y \to \P^3\), and let \(\bp\) be a
  very general set of \(9\) points in \(\P^3\), none of which is
  contained in the branch locus of \(s\). Take the \(q_i\) to be the
  preimages of these \(9\) points under \(s\), so there is a map
  \(s^\prime : Y^\prime \to X_\bp\). If \(\dinf\) is the divisor of
  the previous theorem, then \(s^{\prime\ast}\dinf\) is a movable
  divisor, which has negative intersections with the preimages of each
  of the curves \(C_n\).  As above, it follows that
  \(\Bm(s^{\prime\ast}\dinf)\) is a countable union of curves.
\end{proof}

Though the divisor \(\dinf\) is not big, a standard construction gives
a big \(\R\)-divisor on a smooth 4-fold with non-closed diminished
base locus.  Fix an embedding \(X \to \P^N\), let \(CX \subset
\P^{N+1}\) be the projective cone over \(X\), and take \(p : Y \to
CX\) the blow-up at the cone point.  The map \(p\) is birational with
a unique exceptional divisor \(E \cong X\); write \(i_E : X \to Y\)
for the inclusion.  The variety \(Y\) has the structure of a
\(\P^1\)-bundle \(q : Y \cong \P_X(\cO_X \oplus \cO_X(1)) \to X\).
\begin{lemma}
\label{bignonclosed}
There exists a big \(\R\)-divisor \(\dinf^\prime\) on \(Y\) with
\(\Bm(\dinf^\prime)\) a countable union of curves.
\end{lemma}
\begin{proof}
  Let \(H\) be an ample divisor on \(CX\) with support disjoint from
  the cone point, and set \(\dinf^\prime = p^\ast H + q^\ast \dinf\).
  Choosing \(H\) sufficiently large, we may assume the base locus of
  \(D_\lambda^\prime\) is contained in \(E\). Observe that
  \(D_\lambda^\prime\) is the sum of a big divisor and a
  pseudoeffective one, and thus big.  

  Properties (2), (4), and (5) of Lemma~\ref{bminusbasic} imply that
  \(\Bm(\dinf^\prime) \subseteq \Bm(p^\ast H) \cup \Bm(q^\ast \dinf) =
  \Bm(q^\ast \dinf) = q^{-1}\Bm(\dinf)\).  Furthermore, the choice of
  \(H\) implies that \(\Bm(\dinf^\prime) \subseteq E\), and so
  \(\Bm(\dinf^\prime) \subseteq q^{-1}\Bm(\dinf) \cap E\), which is a
  countable union of curves.  Moreover, each curve \(C_j^\prime =
  i_E(C_j)\) has \(C_j^\prime \cdot \dinf^\prime = q(C_j^\prime) \cdot
  \dinf < 0\), and so \(C_j^\prime \subset \Bm(\dinf^\prime)\). It
  follows that \(\Bm(D_\lambda^\prime)\) is a countable union of
  curves, all contained in \(E\).
\end{proof}

\section{Zariski Decomposition of \(\dinf\)}
\label{zariski}

The non-closedness of \(\Bm(\dinf)\) further implies that \(\dinf\)
admits no Zariski decomposition in several standard senses.  Recall
the form of decomposition in dimension two:
\begin{theorem}[Zariski decomposition theorem, e.g.\ \cite{prokhorov}]
\label{zdecomp}
  Let \(D\) be a pseudoeffective \(\R\)-divisor on a smooth projective
  surface \(X\). There exists an effective divisor \(N = \sum_i a_i
  N_i\) such that \(P = D - N\) is nef, \((N_i \cdot N_j)\) is
  negative definite, and \(P \cdot N_i = 0\).
\end{theorem}

There are several analogues of Zariski decompositions for divisors on
higher-dimensional varieties, imposing conditions which ensure the
retention of useful properties of the two-dimensional version.  One
decomposition which always exists and has proved important is the
divisorial Zariski decomposition of a pseudoeffective \(\R\)-divisor
\(D\), due to Nakayama.
\begin{definition}[\cite{nakayama}]
  Suppose that \(D\) is an \(\R\)-divisor. For a prime divisor \(E\)
  on \(X\), let
\[
\sigma_E(D) = \sup_{\textrm{$A$ ample}}  \left( \min_{D^\prime
     \equiv_{\mathrm{num}} D + A } \ord_E(D^\prime) \right).
\]
Set \(N_\sigma(D) = \sum_E \sigma_E(D) \cdot E\), and \(P_\sigma(D) =
D - N_\sigma(D)\). This is a finite sum, and \(P_\sigma(D) \in
\Movb(X)\).  When \(D\) is a big \(\Q\)-divisor, in fact \(\sigma_E(D)
= \min_{D^\prime \equiv_{\text{num}} D} \ord_E(D^\prime)\).
\end{definition}
In dimension two, this coincides with the standard Zariski
decomposition, but in higher dimensions \(P_\sigma(D)\) is only
movable and not in general nef.  To obtain a closer analogue of the
Zariski decomposition, given a pseudoeffective \(\R\)-divisor on a
smooth variety \(X\), one might ask for a birational modification \(f
: Y \to X\) and a decomposition \(f^\ast D = P + N\), with \(P\)
nef and \(N\) effective.  This is termed a \emph{weak Zariski
  decomposition} by Birkar~\cite{birkar2009}.  One might additionally
ask that:
\begin{enumerate}
\item CKM: the maps \( H^0(Y,\cO_Y(\floor{mP})) \to
  H^0(Y,\cO_Y(\floor{mf^\ast D}))\) are all isomorphisms.
\item Fujita: if \(g : Y^\prime \to Y\) is birational, and \(P^\prime
  \leq g^\ast f^\ast D\) is nef, then \(P^\prime \leq g^\ast P\).
\item Nakayama: \(P = P_\sigma(f^\ast D)\) is the positive part of the
  divisorial Zariski decomposition.
\end{enumerate}
Each of these seeks to extend a property of the usual two-dimensional
Zariski decomposition to the higher-dimensional setting.  The
survey~\cite{prokhorov} of Prokhorov introduces the important
properties of these and other higher-dimensional versions of the
Zariski decomposition.  Nakayama constructed an example of an
\(\R\)-divisor on a \(\P^2\)-bundle over an abelian surface which
admits no Zariski decomposition any of these three
senses~\cite{nakayama}.  However, the divisor of Nakayama's example is
itself big, thus effective, and trivially admits a weak Zariski
decomposition. However, we show \(\dinf\) does not admit a weak
Zariski decomposition, and that \(\dinf^\prime\) of
Lemma~\ref{bignonclosed} is another example of big divisor with no
decomposition in the sense of Nakayama.

\begin{lemma}[$=$ Theorem~\ref{main}, (iii)]
\label{nowzd}
\(\dinf\) does not admit a weak Zariski decomposition, and
\(\dinf^\prime\) does not admit a Zariski decomposition in the sense
of Nakayama.
\end{lemma}
\begin{proof}
  Suppose that \(f^\ast \dinf = P + N\) where \(N\) is effective.  For
  each \(n\), pick a curve \(\tilde{C}_n\) on \(Y\) mapping finitely
  to \(C_n\), and let \(d_n = \deg(\tilde{C}_n \to C_n)\).  Only
  finitely many of the \(\tilde{C}_n\) are contained in \(\Supp N\),
  since these curves are Zariski dense.  On the other hand, for any
  curve \(\tilde{C}_n\) not contained in \(\Supp N\), we have
  \(\tilde{C}_n \cdot N \geq 0\), and so compute \(d_n (\dinf \cdot
  C_n) = f^\ast \dinf \cdot \tilde{C}_n = P \cdot \tilde{C}_n + N
  \cdot \tilde{C}_n \geq 0\), a contradiction.  Similarly, the
  non-closedness of \(\Bm(\dinf^\prime)\) implies this divisor does
  not admit a Zariski decomposition in the sense of
  Nakayama~\cite[pg.\ 28]{nakayama}.
\end{proof}

\section{Acknowledgments}

I am indebted to my advisor, James M\textsuperscript{c}Kernan, for
many useful discussions and comments, and to the anonymous referees,
who suggested some substantial improvements.  Thanks also to Mihai
Fulger, Mircea Musta{\c{t}}{\u{a}}, and Rob Lazarsfeld for helpful
suggestions, and to Igor Dolgachev, who kindly directed me to a number
of useful sources on the Cremona action.  I also benefited greatly
from discussions with Roberto Svaldi and Tiankai Liu.


\bibliographystyle{amsplain}
\bibliography{zrefs}

\end{document}